\documentclass{amsart}
\usepackage{graphicx}
 \makeatletter
\renewcommand*\subjclass[2][2000]{%
  \def\@subjclass{#2}%
  \@ifundefined{subjclassname@#1}{%
    \ClassWarning{\@classname}{Unknown edition (#1) of Mathematics
      Subject Classification; using '1991'.}%
  }{%
    \@xp\let\@xp\subjclassname\csname subjclassname@#1\endcsname
  }%
}
 \makeatother

\newtheorem{theorem}{Theorem}[section]
\newtheorem{lemma}[theorem]{Lemma}
\newtheorem{corollary}[theorem]{Corollary}
\newtheorem{proposition}[theorem]{Proposition}
\theoremstyle{definition}

\numberwithin{equation}{section}
\begin{document}

\title[Quasiconformal harmonic maps between surfaces]{On certain nonlinear elliptic PDE and
quasiconfomal maps between Euclidean surfaces} \subjclass{Primary
30C65 ; Secondary 53C43, 30C35}

%\date{11 October, 2005}
\author{David Kalaj}
\address{University of Montenegro, Faculty of Natural Sciences and
Mathematics, Cetinjski put b.b. 81000 Podgorica, Montenegro}
\email{davidk@t-com.me} % \maketitle
\author{Miodrag Mateljevi\'c}
\address{University of Belgrade, Faculty of Mathematics, Studentski
trg 16, 11000 Belgrade, Serbia}
\email{miodrag@matf.bg.ac.rs}\maketitle

\begin{abstract}
%%%A
 We  mainly investigate some properties of quasiconformal
mappings between smooth $2$-dimensional surfaces with boundary in
the Euclidean space, satisfying certain partial differential
equations (inequalities) concerning Laplacian, and  in particular
satisfying Laplace equation and show that that these mappings are
Lipschitz. Conformal parametrization of  such  surfaces and  the
method developed in our paper \cite{km}  have important role in this
paper.

\end{abstract}

\section{Introduction}% and auxiliary results}
By $\Bbb U$ we denote the unit disk, by $\Omega$ a domain in $\Bbb
R^2$ and by $S$ a smooth 2-dimensional surface in $\Bbb R^l$, $l\ge
3$.
%Let $f$ be a homeomorphism
%between two surfaces $S_1\subset \Bbb R^n$ and $S_2\subset\Bbb
%R^m$. We say that $f$ is $K$ ($K\ge 1$) {\it quasi-conformal} or
%shortly q.c. mapping if $$\limsup_{r\to 0}
%\frac{\sup\{d_m(f(x),f(y)):d_n(x,y)=r, x,y\in S_1
%\}}{{\inf\{d_m(f(x),f(y)):d_n(x,y)=r, x,y\in S_1\}}}\le K.$$

%%% A1
Let $f$ be a smooth mapping between a Jordan domain $\Omega$ and a
surface $S$ of the Euclidean space $\Bbb R^l$. Consider the
functional
\begin{equation}\label{one}E[f]=\iint_{\Omega} |f_x|^2+|f_y|^2 dx
dy\end{equation}
%and minimize it under the boundary conditions $f|_{S^1}(z)=\gamma(z)\in \Gamma=\partialS$.
The stationary points of the energy integral $E[f]$ satisfy the
Euler-Lagrange equation i.e. Laplace equation
\begin{equation}\label{ele}
\Delta f=f_{xx}+f_{yy}=4f_{z\overline z}\equiv 0.
\end{equation}
The mapping $f$ satisfying the relation (\ref{ele}) is called {\it
harmonic}.
%%%%

%%%B1

%If $\varphi$ is a conformal mapping of $D$ into $\Omega$ then
%$f\circ \varphi$ is also harmonic.

%%% A0
%%%%
Let us define harmonic mappings and quasiconformal mappings between
two smooth 2-dimensional surfaces $S_1\subset \Bbb R^m$ and
$S_2\subset \Bbb R^n$. For every $a\in S_1$ let $X_a(x,y)$ be a {\it
conformal mapping} between the unit disk and a neighborhood
$U_a\subset S_1$ i.e. let $x,y$ be {\it isothermal coordinates} in
$U_a$. The mapping $f$ of the surface $S_1$ into the surface $S_2$
is called {\it harmonic} if for every $a\in S_1$ $f\circ X_a : \Bbb
U \rightarrow  \Bbb R^n$ is harmonic in $\Bbb U$. Let $Y=Y_{f(a)}$
be isothermal coordinates in some neighborhood $V_{f(a)}$ in $Y$. It
means that {\bf }$g=(Y_{f(a)})^{-1}\circ f\circ X_a$  is
$\rho-$harmonic, where $\rho(w)=|Y_u(w)|^2$; we also say that $f$ is
harmonic
 with respect to the metric on $S_2$
inherited from the Euclidean space  $\Bbb R^n$. Let $k\in [0,1)$ and
let $f$ be a homeomorphism between $S_1$ and $S_2$. Let $a\in S^1$
be arbitrary and let $X_a$ be isothermal coordinates in $U_a$.
Similarly let $Y_{f(a)}$ be isothermal coordinates in some
neighborhood $V_{f(a)}$ in $Y$. If for every $a$ the mapping
$g=(Y_{f(a)})^{-1}\circ f\circ X_a$ satisfies the inequality
$|g_{\bar z}|\le k |g_z|$ in $X_a^{-1}(U_a)$, then $f$ is said to be
a $k$ {\it quasiconformal (q.c.) mapping}.
%%%%
%%%%

The domain $\Omega$ (the surface $S$) is called $C^{l,\alpha}$
domain (surface) if the boundary $\partial \Omega$ ($\partial S$) is
a compact $C^{l,\alpha}$ 1-dimensional manifold (curve).

%%%%%

%%%% remark
%%%%

%%%

%%%
In this paper we continue to study the boundary behaviors of q.c.
harmonic mappings between plane domains and Euclidean surfaces.
Notice this important fact, the class of q.c. harmonic mappings
contains conformal mappings (see the section below for boundary
behaviors of conformal mappings).

Further developments of the method presented in \cite{km} leads to
Theorem~\ref{qcie} and Theorem~\ref{qcie1}: it is proved that every
q.c. $C^2$ diffeomorphism $w$ between two plane domains with smooth
boundaries satisfying the inequality
\begin{equation}\label{pdi}|\Delta w|\le M|\nabla
w|^2+N\end{equation} is Lipschitz continuous. The inequality
\eqref{pdi} we will call a {\it Poisson differential inequality}.
These theorems imply corresponding results for q.c. harmonic
mappings between smooth surfaces (Theorem~\ref{theom},
Theorem~\ref{theom1}). This extends the results of the authors
\cite{km} where instead of \eqref{pdi} is assumed that
\begin{equation}\label{pdi0}|\Delta w|\le M| w_z w_{\overline
z}|.\end{equation}

For the background on the theory of q.c. harmonic mappings in the
plane we refer to the papers   \cite{Kalajpub}--\cite{KP},
\cite{Om}--\cite{ps}.

%%%
%%%
\section{Conformal parametrization}\label{cop}
\begin{proposition}[Kellogg and Warshawski see
\cite{Ko}, \cite{w1} and \cite{w2}]\label{oneone} Let $l\in \Bbb N$,
$0<\alpha\le 1$. If $\Omega$ and $D$ are Jordan domains having
$C^{l,\alpha}$  boundaries and $\omega$ is a conformal mapping of
$\Omega$ onto $D$, then $\omega^{(l)}\in C^{\alpha}(\overline
\Omega).$ In particular $\omega^{(l)}$ is bounded from above on
$\Bbb U$.
\end{proposition}

The following theorem can be viewed as an extension of
Proposition~\ref{oneone} and of Riemann mapping theorem.
\begin{theorem}\cite[Theorem~3.1]{jost}\label{theo} Suppose $S$ is a surface with boundary,
homeomorphic to a plane domain $G$ bounded by $k$ circles via a
chart $\psi: \overline G\mapsto S$. Suppose the coefficients of the
metric tensor of $S$ can be defined in this chart by bounded
measurable functions $g_{ij}$  with $g_{11} g_{22} - g_{12}^2\ge
\lambda >0$ in $G$. Then $S$ admits a conformal representation
$\tau\in H_1^2\cap C^{\alpha}(\bar B, \bar G)$, where $B$ is a plane
domain bounded by $k$ circles and $\tau$ satisfies almost everywhere
the conformality relations $$|\tau_x|^2 = |\tau_y|^2 \ \text{and}
\left <\tau_x,\tau_y \right> = 0$$ (Here $(x,y)$ denote the
coordinates of points in $B$, and norms and products are taken with
respect to the metric of $S$).

$\tau$ can be normalized by a three point condition, namely three
points on one of the boundary curves of $S$ can be made to
correspond, respectively, to three given points on the outer
boundary of $B$ which can be taken as the unit circle, or by fixing
the image of an interior point. Furthermore, concerning higher
regularity, $\tau$ is as regular as $S$ , i.e. if $S$ is of class
$C^{m,\alpha}(\bar B)$ ($m\in \Bbb N$, $0<\alpha<1$) or in
$C^\infty$ then also $\tau \in C^{m,\alpha}(\bar B)$ or $\tau \in
C^{\infty}(\bar B)$, respectively. In particular, if $S$ is at least
$C^{1,\alpha}$ then the conformality relations are satisfied
everywhere, and $\tau$ is a diffeomorphism.
\end{theorem}

We will make use the following corollary of the previous theorem.
\begin{corollary}\label{coro}
Let $X:\Bbb U\mapsto S$ be a conformal mapping between the unit disk
and a $C^{2,\alpha}$ surface $S$. Then
\begin{eqnarray}\label{in}&c:=\min_{\{z:|z|\le
1\}}{|X_u(z)|}=\min_{\{z:|z|\le 1\}}{|X_v(z)|}>0,\\&C:=\max_{\{z:
|z|\le 1\}}{|X_{uu}(z)|+|X_{uv}(z)|+|X_{vv}(z)|}<\infty,\text{ and }
\\ &|{\log |X_u(w)|^2}_w|\le M'<\infty.
\end{eqnarray}
\end{corollary}
\begin{proof}
The first two inequalities follow directly from Theorem~\ref{theo}.
For $\rho=\log |X_u(w)|^2$ we have

$$\rho_w=\frac{<X_{uu},X_u>-i<X_{uv},X_u>}{|X_u(w)|^2}=\frac{<X_{uu},X_u>+i<X_{uu},X_v>}{|X_u(w)|^2}.$$
Consequently:
\begin{equation}\label{ineq}
|\rho_w|\le 2 \frac{|X_{uu}|}{|X_u|}\le M'=\frac Cc.
\end{equation}
\end{proof}

\section{The main results}
Firstly we are going to establish a local Lipschitz character of our
mappings.
\begin{theorem}[The main theorem]\label{qcie}
Let $f$ be a quasiconformal $C^2$ diffeomorphism from the plane
domain $\Omega$ onto the plane domain $G$. Let $\gamma_\Omega\subset
\partial \Omega$ and $\gamma_G=f(\gamma_\Omega)\subset \partial G$ be $C^{1,\alpha}$ respectively $C^{2,\alpha}$ Jordan arcs.
If for some $\tau\in \gamma_\Omega$ there exist positive constants
$r$, $M$ and $N$ such that
\begin{equation}\label{newinequ1}
|\Delta f|\le M|\nabla f|^2+N\,, \quad  z\in \Omega\cap
D(\tau,r),\end{equation} then $f$ has bounded partial derivatives in
$\Omega\cap D(\tau,r_\tau)$ for some $r_\tau<r$. In particular it is
a Lipschitz mapping in $\Omega\cap D(\tau,r_\tau)$.
\end{theorem}
We need the following proposition.
\begin{proposition}{\bf (Heinz-Bernstein, see
\cite{EH})}. \label{heb}
 Let $s: \overline{\mathbb{U}} \to
\mathbb{R}$ be a continuous function from the closed unit disc
$\overline{\mathbb{U}}$ into the real line satisfying the
conditions:
\begin{enumerate}\item $s$ is $C^2$ on ${\mathbb{U}}$,
\item $s_b(\theta)= s(e^{i\theta})$ is $C^2$ and
\item $ |\Delta s| \leq M_0 |\nabla s|^2+N_0$, on $\mathbb{U}$ for some
constants $M_0$ and $N_0$.
\end{enumerate} Then the function $|\nabla s|$ is bounded on $\mathbb{U}$.
\end{proposition}
\begin{proof}[\it{Proof of Theorem~\ref{qcie}}]
Let $r>0$ be sufficiently small positive real number such that
$\Delta=D(\tau,r)\cap \Omega$ is a Jordan domain with $C^{1,\alpha}$
boundary consisting of a circle arc $C(t_0,t_1)$ and an arc
$\gamma_0[t_0,t_1]\subset \gamma$ containing $\tau$. Take
$D=f(\Delta)$. Let $g$ be a conformal mapping of the unit disc onto
$\Delta$. Let $\tilde f=f\circ g$. Since $\Delta \tilde
f=|g'|^2\Delta f$ and $|\nabla \tilde f|^2=|g'|^2|\nabla f|^2$, we
find that, $\tilde f$ satisfies the inequality (\ref{newinequ1})
with $M_1=M$ and $N_1=N\cdot{\inf_{|z|\le 1}|g'(z)|}^{-1}$. We will
prove the theorem for $\tilde f$ and then apply Kellogg's theorem.
For simplicity, we write $f$ instead of $\tilde f$. Let $J$ be a
compact subset of $\gamma_0$ containing $\tau$ but not containing
the points $t_0$ and $t_1$. Let $t\in J$ be arbitrary.

 Step 1  ({\bf Local Construction}). In this step we show that
there are two Jordan domains $D_1$ and $D_2$  in $D$ with
$C^{2,\alpha}$ boundary such that
\begin{enumerate}
\item[(i)] $D_1\subset D_2\subset D$,

\item[(ii)] $\partial D\cap \partial D_2$ is a connected arc
containing the point $w=f(t)$ in its interior,

\item[(iii)] $\emptyset\neq \overline{\partial D_2\setminus \partial
D_1}\subset D$.
\end{enumerate}

Let $H_1$ be the Jordan domain bounded by the Jordan curve
$\gamma_1$ which is composed by the following sequence of Jordan
arcs: $\{y^{1/5}+(2-x)^{1/5}=1,\, 1\le x\le 2\}$;
$\{(2-y)^{1/5}+(2-x)^{1/5}=1,\, 1\le x\le 2\}$; $[(1,2),(-1,2)]$;
$\{(2-y)^{1/5}+(2+x)^{1/5}=1,\, -2\le x\le -1\}$;
$\{y^{1/5}+(2+x)^{1/5}=1,\, -2\le x\le -1\}$ and $[(-1,0),(1,0)]$.
Let $H_2$ be the Jordan domain bounded by the Jordan curve
$\gamma_2$ which is composed by the following sequence of Jordan
arcs: $\{y^{1/5}+(2-x)^{1/5}=1,\, 1\le x\le 2\}$; $[(2,1),(2,2)]$;
$\{(3-y)^{1/5}+(2-x)^{1/5}=1,\, 1\le x\le 2\}$; $[(1,3),(-1,3)]$;
$\{(3-y)^{1/5}+(2+x)^{1/5}=1,\, -2\le x\le -1\}$; $[(-2,2),
(-2,1)]$; $\{y^{1/5}+(2+x)^{1/5}=1,\, -2\le x\le -1\}$ and
$[(-1,0),(1,0)]$. Note that $H_1\subset
H_2\subset[-2,2]\times[0,3]$, $\partial H_1\cap \Bbb R=\partial
H_2\cap \Bbb R=[-1,1]$ and that $\partial H_1,\partial H_2\in
C^{3}$.

Let $\Gamma$ be an orientation preserving arc-length
parameterization of $\gamma=\partial D$ such that for
$s_0\in(0,\mathrm{length}(\gamma))$ there holds: $\Gamma(s_0)=f(t)$.
Let $D^*=\overline{\Gamma'(s_0)}D$, $b=\overline{\Gamma'(s_0)}f(t)$
and $\Gamma^*=\overline{\Gamma'(s_0)}\Gamma$. Then there exists
$r>0$ such that $(b,b+ir]\subset D^*$. Since $\gamma^*=\partial
D^*\in C^{2,\alpha}$, it follows that, there exist $x_0>0$,
$\varepsilon>0$, $y_0\in(0,r/3)$, the $C^{2,\alpha}$ function
$h:[-2x_0,2x_0]\to \Bbb R$, $h(0)=0$, and the domain $D_2^*\subset
D^*$ such that:
\begin{enumerate}
\item
$\Gamma^*([s_0-\varepsilon,s_0+\varepsilon])=\{b+(x,h(x)):
\,x\in[-2x_0,2x_0]\}$,
\item
$D_2^*=\{b+(x,h(x)+y): \,x\in[-2x_0,2x_0],\, y\in(0,3y_0]\}$.
\end{enumerate}

Let $\Upsilon:[-2,2]\times[0,3]\to D_2^*$ be the mapping defined by:
$$\Upsilon(x,y)=b+(xx_0,h(xx_0)+yy_0).$$ Then $\Upsilon$ is a
$C^{2,\alpha}$ diffeomophism.

Take $D_i=\Gamma'(s_0)\cdot\Upsilon(H_i)$, $i=1,2$. Obviously
$D_1\subset D_2\subset D$ and $D_1$ and $D_2$ have $C^{2,\alpha}$
boundary. Observe that $f(t)=\Gamma'(s_0)\overline{\Gamma'(s_0)}
f(t)=\Gamma'(s_0)\Upsilon(0)\in \Gamma'(s_0)\Upsilon([-1,1])=
\partial D_1\cap \partial D_2.$

Step 2 ({\bf Application of Heinz-Bernstein theorem}).   Let $\phi$
be a conformal mapping of $D_2$ onto $H$ such that
$\phi^{-1}(\infty)\in{\partial D_2\setminus
\partial D_1}$. Let $\Omega_1=\phi(D_1)$. Then there exist real
numbers $a,$ $b$, $c$, $d$ such that $a<c<d<b$, $[a,b]=\partial
\Omega_1\cap \Bbb R$ and $l=\phi^{-1}(\partial \Omega_1 \setminus
[c,d])\subset D$. Let $U_1=f^{-1}(D_1)$ and $\eta$ be a conformal
mapping between the unit disc and the domain $U_1$. Then the mapping
$\hat f=\phi\circ f\circ \eta$ is a $C^2$ diffeomorphism of the unit
disc onto the domain $\Omega_1$ such that:
\begin{enumerate}
\item[(a)] $\hat f$ is continuous on the boundary $\Bbb T=\partial \Bbb U$ (it is
q.c.) and

\item[(b)] $\hat f$ is $C^2$ on the set $T_1=\hat f^{-1}(\partial
\Omega_1 \setminus (c,d))$.
\end{enumerate}

Let $s:=\mathrm{Im}\, \hat f$. First, note that (a) implies that $s$
is continuous on $\Bbb T=\partial \Bbb U$. On other hand, as $\hat
f\in C^2$, $s$ satisfies the condition:
\begin{enumerate}
\item[(1)] $s\in C^2(\Bbb U)$.
\end{enumerate}
From (b) we obtain that $s$ is $C^2$ on the set $T_1=\hat
f^{-1}(\partial \Omega \setminus (c,d))$. Furthermore, $s=0$ on
$T_2=\hat f^{-1}(a,b)$;  and therefore $s$ is $C^2$ on $T_2=\hat
f^{-1}(a,b)$. Hence:
\begin{enumerate}
\item[(2)] $s$ is $C^2$ on $\Bbb T=T_1\cup T_2$. In other words, the function $s_b:\Bbb R\to
\Bbb R$ defined by $s_b(\theta)=s(e^{i\theta})$ is $C^2$ in $\Bbb
R$.
\end{enumerate}
In order to apply the interior estimate, we have to prove that
\begin{enumerate}
\item[(3)] $|\Delta s(z)|\le M_0|\nabla s(z)|^2+N_0$, $z\in \Bbb U,$  where
$M_0$ and $N_0$ are constants.
\end{enumerate}
To continue we need the following lemma:
\begin{lemma}\label{lemah}
If $f=u+iv$ is a q.c. mapping satisfying Poisson differential
inequality, then $u$ and $v$ satisfy the Poisson differential
inequality.
\end{lemma}
\begin{proof}
Let $$A:=|\nabla u|^2=2(|u_z|^2+|u_{\bar z}|^2)=\frac
12(|f_z+\overline{f_{\bar z}} |^2+|f_{\bar z}+\overline{f_z}|^2)$$
and $$B:=|\nabla v|^2=2(|v_z|^2+|v_{\bar z}|^2)=\frac
12(|f_z-\overline{f_{\bar z}} |^2+|f_{\bar z}-\overline{f_z}|^2).$$
Then
$$\frac{A}{B}=\frac{|1+\mu|^2}{|1-\mu|^2}$$ where
$\mu={\overline{f_{\bar z}}}/{f_z}$. Since $|\mu|\le k$
\begin{equation}\label{help} \frac{(1-k)^2}{(1+k)^2}\le \frac AB
\le \frac{(1+k)^2}{(1-k)^2}.
\end{equation}
As $$|\Delta f|=|\Delta u+i\Delta v|\le M|\nabla f|^2+N=M(|\nabla
u|^2+|\nabla v|^2)+N,$$  the relation (\ref{help}) yields
$$|\Delta u|\le M\frac{(1+k)^2}{(1-k)^2}|\nabla u|^2+N$$ and
$$|\Delta v|\le M\frac{(1+k)^2}{(1-k)^2}|\nabla v|^2+N.$$
\end{proof}
Since $\hat f=\phi\circ f\circ \eta$, we obtain
\begin{equation}\label{posht}\partial \hat f=\phi'\partial f \eta',\,\,\,\,\,\bar
\partial \hat f=\phi'\bar \partial f \eta'\end{equation} and
\begin{equation}\label{nalt}\partial \bar \partial \hat f=\frac 14\Delta \hat
f=\frac 14\Delta(\phi\circ f)\cdot |\eta'|^2=(\phi''\partial f\cdot
\bar\partial f+\phi'\partial \bar\partial f)|\eta'|^2.\end{equation}
Now combining (\ref{newinequ1}), (\ref{posht}) and (\ref{nalt}) we
obtain
\[\begin{split}|\Delta \hat f|&\le
4\frac{|\phi''|}{|\phi'|^2}|\partial \hat f||\bar\partial\hat
f|+|\phi'||\Delta f||\eta'|^2\\&\le
4\frac{|\phi''|}{|\phi'|^2}|\partial \hat f||\bar\partial\hat
f|+|\phi'|\left(M|\nabla f|^2+N\right)|\eta'|^2 \\&
\le\frac{|\phi''|}{|\phi'|^2}|\nabla \hat f|^2+M|\nabla \hat
f|^2\cdot
\frac{1}{|\phi'|}+N|\phi'||\eta'|^2\\&=\left(\frac{|\phi''|}{|\phi'|^2}+\frac{M}{|\phi'|}\right)|\nabla
\hat f|^2+N|\phi'||\eta'|^2.\end{split}\]
As $\hat f$ is a $k$- q.c.
mapping using Lemma~\ref{lemah} we have
\begin{equation}\label{nin} |\Delta s|\le \frac{(1+k)^2}{(1-k)^2}\left(\frac{|\phi''|}{|\phi'|^2}+\frac{M}{|\phi'|}\right)
\cdot |\nabla s|^2+N|\phi'||\eta'|^2.
\end{equation}

Proposition~\ref{oneone} implies that the function $|\eta'|$ is
bounded from above by a constant $C_1$, the function $|\phi'|$ is
bounded from below and above  by positive constants $C_2$ and $C_3$
respectively and the function $|\phi''|$ is bounded from above by a
constant $C_4$. Hence $$|\Delta s|\le M_0|\nabla s|^2+N_0,$$ where
$$M_0=\frac{(1+k)^2}{(1-k)^2}\left(\frac{C_4}{C^2_3}+\frac{M}{C_3}\right)\text{
and }N_0=C_2C_1^2N.$$

Proposition~\ref{heb} implies that, the function $|\nabla s|$ is
bounded by a constant $b_t$. Since ${{\hat f}}$ is a $k-$q.c.
mapping, we have
$$(1-k)|\partial{{\hat f}}|\le |\partial{{{\hat
f}}-\overline{\bar\partial{{\hat f}}}}|\le 2|s_z|\le\sqrt 2 b_t.$$
Finally, $$|\partial{{\hat f}}|+|\bar\partial{{\hat f}}|\le \sqrt
2\frac {1+k}{1-k} b_t.$$

%Let $T^t_{c,d}=(f\circ \eta)^{-1}(c,d)$. Observe that $c$ and $d$
%depend on the fixed point $t$. Since $t\in T^t_{c,d}$ we obtain
%$\Bbb T=\bigcup_{t\in \Bbb T}T^t_{c,d}$, and therefore there
%exists a finite set $\{t_1,\dots, t_n\}$ such that $\Bbb
%T=\bigcup_{i=1}^nT^{t_i}_{c,d}$.

Since the mapping $\eta$ is conformal and maps the circle arc
$T=(\phi\circ f\circ \eta)^{-1}(a,b)$ onto the circle arc
$(\phi\circ f)^{-1}(a,b)$, it follows that, it can be conformally
extended across the arc $T'=(\phi\circ f\circ \eta)^{-1}[c,d]$.
Hence, there exists a constant $A$ such that $|\eta'(z)|\ge 2A$ on
$T'$. It follows that there exists $r\in(0,1)$ such that
$|\eta'(z)|\ge A$ in $\{\rho z: z\in T', r\le\rho\le 1\}$. It
follows from the Proposition~\ref{oneone}, that the conformal
mapping $\phi$ and its inverse have the $C^1$ extension to the
boundary. Therefore there exists a positive constant $B$ such that
$|\phi'(z)|\ge B$ on some neighborhood of $\phi^{-1}[c,d]$ with
respect to $D$. Thus, the mapping $f=\phi^{-1}\circ \hat f\circ
\eta^{-1}$ has bounded derivative in some neighborhood of the set
$\eta(T')$, on which it is bounded by the constant $$C= \sqrt 2\frac
{1+k}{1-k} \frac{b_ {t}}{AB}.$$ Then $$|\partial f(z)|+|\bar
\partial f(z)|\le C_0\text{ for all $z\in \Bbb U$ near the arc
$T=\eta(T')$}.$$
\end{proof}
\begin{theorem}\label{qcie1}
Let $f$ be a quasiconformal $C^2$ diffeomorphism from the plane
domain $\Omega$ with $C^{1,\alpha}$ compact boundary onto the plane
domain $G$ with $C^{2,\alpha}$ compact boundary. If there exist
constants $M$ and $N$ such that
\begin{equation}\label{newinequ2}
|\Delta f|\le M|\nabla f|^2+N\,, \quad  z\in \Omega ,\end{equation}
then $f$ has bounded partial derivatives in $\Omega$. In particular
it is a Lipschitz mapping in $\Omega$.
\end{theorem}
\begin{proof}
According to the Theorem~\ref{qcie} for every $t\in
\partial\Omega$ there exists $r_t>0$ such that $f$ has bounded
partial derivatives in $\Omega\cap D(t,r_t)$. Since $\partial
\Omega$ is a compact set it follows that there exist $t_1,\dots,t_m$
such that $\partial \Omega\subset \bigcup_{i=1}^m D(t_i,r_{t_i})$.
It follows that $f$ has bounded partial derivatives in
$\Omega\cap\bigcup_{i=1}^m D(t_i,r_{t_i})$. Since $f$ is a
diffeomorphism in $\Omega$, we obtain that $f$ has bounded
derivatives in the compact set $\Omega\setminus \bigcup_{i=1}^m
D(t_i,r_{t_i})$. The conclusion of the theorem now easily follows.
\end{proof}
\begin{corollary}
Let $\Omega$ be a plane domain with $C^{1,\alpha}$ compact boundary
and $G$ be a plane domain with $C^{2,\alpha}$ compact boundary. If
$w=f(z):\Omega\mapsto G$ is a quasiconformal solution of the
equation
\begin{equation}\label{pdee}\begin{split}\alpha
w_{xx}&+2\beta w_{xy}+\gamma w_{yy}+a_1(z)w^2_x+b_1(z)w_x
w_y+c_1(z)w_y^2\\
&+a(z)w_x+b(z)w_y+c(z)w+d(z)=0,\end{split}
\end{equation} such that
$\alpha,\beta,\gamma\in \Bbb R$, $\alpha>0$,
$\alpha\gamma-\beta^2>0$, $a,b,c,d,a_1,b_1,c_1\in C({\overline
\Omega})$, then $f$ is Lipschitz.
\end{corollary}
\begin{proof}
Since the partial differential equation (PDE) (\ref{pdee}) is
elliptic, we can choose coordinates $x=\alpha_1 u+\beta_1 v$,
$y=\beta_1 u+\gamma_1 v$ such that (\ref{pdee}) becomes
\begin{equation}\label{begeq}\begin{split}w_{uu}&+w_{vv}+a_1'(u,v)w^2_u+b_1'(u,v)w_u
w_v+c_1'(u,v)w_v^2\\&+a'(u,v)w_u+b'(u,v)w_v+c'(u,v)w+d'(u,v)=0,
\,\,(u,v)\in\Omega'.\end{split}\end{equation} For $e\in
C(\overline{\Omega'})$ let
$|e|=\max\{|e(u,v)|:(u,v)\in\overline{\Omega'}\}$. Using
(\ref{begeq}) and the inequality $2|t|\le |t|^2+1$ we obtain
\[\begin{split}|\Delta w|&\le \left(\frac{|a'|}{2}+\frac{|b'|}{2}\right)(|\nabla w
|^2+1)+\left(\max\{|a_1'|,|c_1'|\}+\frac{|b_1'|}{2}\right)|\nabla w
|^2+|c'||w|+|d'|\\&=M|\nabla w|^2+N,\end{split}\] where
$$M=(|a'|+|b'|)/2+\max\{|a_1'|,|c_1'|\}+\frac{|b_1'|}{2}$$ and
$$N=\frac{|a'|+|b'|}{2}+|c'||w|+|d'|.$$ The conclusion now follows
from Theorem~\ref{qcie1}.
\end{proof}

% $d$ is  the Euclid distance
%If the minimum occur it follows that the mapping
%$g(z)=X^{-1}(f(z))$ satisfies the equation (\ref{el}) if and only
%if $f$ satisfies the equation (\ref{ele}).
%\begin{definition} We call a surface $S^2$ quasi-minimal if there
%exists a harmonic q.c. parametrisation $f$ of $S^2$.
%\end{definition}

%%%% B
 By  $d=d_k$  we denote Euclidean distance in Euclidean space $\Bbb R^k$.

\begin{theorem}\label{theom} We call a $C^{2,\alpha}$ surface $S$  disk-like surface
if it is homeomorphic to the unit disk, and if its boundary is a
$C^{2,\alpha}$ curve. If $f$ is a quasiconformal harmonic mapping
between two $C^{2,\alpha}$ disk-like surfaces $S_1$ and $S_2$, then
it is a Lipschitz mapping i.e. there exists a constant $C$ such that
$$d(f(x),f(y))\le Cd(x,y), \text{ for all $x,y\in S_1$}.$$
\end{theorem}
\begin{proof}
Let $f$ be a harmonic q.c. mapping between disk-like surfaces $S_1$
and $S_2$. Let $X:\Bbb U\mapsto S_1$ and $Y:\Bbb U \mapsto S_2$ be
conformal mappings. Let us consider the mapping $g=Y^{-1}\circ
f\circ X$ of the unit disk onto itself. Since $f(X(z))=Y(g(z))$, it
follows that $$|f\circ X_x|^2+|f\circ X_y|^2=|Y_u|^2
(|g_x|^2+|g_y|^2).$$

Hence \begin{equation}\label{two} E[f\circ X]=E_Y[g]=\iint_{\Omega}
|Y_u|^2(|g_x|^2+|g_y|^2) dx dy.
\end{equation}

If we denote $\rho(w)=|Y_u(w)|^2$, then the stationary points of the
energy integral $E_Y[g]$ satisfies the Euler-Lagrange equation
\begin{equation}\label{el}
g_{z\overline z}+{(\log \rho)}_w\circ g g_z\,g_{\bar z}=0.
\end{equation}
Consequently $f\circ X$ is harmonic if and only if $g$ is
$\rho-$harmonic i.e. the mapping satisfying the relation (\ref{el}).
According to the Corollary~\ref{coro}, the mapping $g$ satisfies the
conditions of Theorem~\ref{qcie1}. Namely as $|{(\log \rho)}_w|\le
M$ and $|g_zg_{\bar z}|\le 1/2(|g_z|^2+|g_{\bar z}|^2)$ we can
simply take $M=M'/2$, and $N=0$. Theorem~\ref{qcie1} yields that $g$
is Lipschitz. By Theorem~\ref{theo} it follows that $X$ and $Y$ are
bi-Lipschitz mappings. $f$ is Lipschitz as a composition of
Lipschitz mappings.
\end{proof}
%%%%
%%%%

Using Theorem~\ref{theom} we obtain the theorem:
\begin{theorem}\label{theom1}
If $f$ is a quasiconformal harmonic mapping between $C^{2,\alpha}$
surfaces $S_1$ and $S_2$, with $C^{2,\alpha}$ compact boundaries
then it is a Lipschitz mapping i.e. there exists a constant $C$ such
that $d(f(x),f(y))\le Cd(x,y), \text{ for all $x,y\in S_1$}.$
\end{theorem}

\subsection{A question.} Recently in \cite{kalan} is proved that, every quasiconformal harmonic mapping between two Jordan
domains $\Omega_1\in C^{1,\alpha}$ and $\Omega_2\in C^{2,\alpha}$ is
bi-Lipschitz, and this can be extended directly to all
$C^{j,\alpha}$, $j=1,2$ plane domains. On the other hand, this
result has been extended in \cite{surfaces} to the $C^{2,\alpha}$
surface with approximately analytic metrics. The question arises,
whether the previous statement can be extended to $C^{1,\alpha}$
surfaces?

\end{document}